\documentclass[11pt]{opelsarticle}
\usepackage{graphicx,tikz,pgfplots, tikz-3dplot, }
\usepackage[colorlinks=true, citecolor=red]{hyperref}
\usepackage{amsmath,amssymb,amsfonts,MnSymbol, mathrsfs}

 \usepackage{mathtools}
 \usepackage{algorithm}
\usepackage{xcolor}
\usepackage{listings}
\usepackage[colorlinks=true, citecolor=red]{hyperref}

\usepackage{tikz,pgfplots,bm}
\usetikzlibrary{patterns}

\numberwithin{equation}{section}

\newtheorem{theorem}{Theorem}
\newtheorem{proposition}[theorem]{Proposition}%
\newtheorem{remark}{Remark}%
\newtheorem{notation}{Notation}
\newtheorem{orientation}{Orientation}%

\newtheorem{lemma}{Lemma}

\def\e{{\rm e}}
\def\ii{{\rm i}}

\def\vE{{\mathbf{E}}}
\def\vS{{\mathbf{S}}}
\def\vI{{\mathbf{I}}}

\def\vF{{\mathbf{F}}}

\def\vX{{\mathbf{X}}}
\def\vY{{\mathbf{Y}}}
\def\vZ{{\mathbf{Z}}}

\def\R{{\mathbb{R}}}
\def\S{{\mathbb{S}}}
\def\I{{\mathbb{I}}}

\def\Z{{\mathbb{Z}}}

\newcommand{\ds}{\displaystyle}

\def\p{\partial}
\def\n{\nabla}

\def\d{\hbox{d}}
\def\One{{\bf 1}}

\def\eq#1{\begin{equation}#1\end{equation}}

\def\eeqn#1{\begin{eqnarray*}#1\end{eqnarray*}}

\def\vx{{\bf x}}

\def\vom{{\bm\omega}}
\def\SS{{\mathbb{S}_2}}

\begin{document}

\title{\textbf{ 
Numerical Simulation of Polarized Light and Temperature with a Refractive Interface}}
\author{Olivier Pironneau\footnote{\emph{olivier.pironneau@sorbonne-universite.fr }, LJLL, Sorbonne Universit\'e, Paris, France.}}

\begin{frontmatter}
\begin{abstract}
In this article we propose a numerical algorithm to compute the intensity and polarization of a polychromatic electromagnetic radiation crossing a medium with graded refractive index and modeled by the Vector Radiative Refractive Transfer Equations (VRRTE). Special attention is given to the case where the refractive index has a discontinuity for which the Fresnel conditions are necessary. We assume that the only spatial variable of interest is the altitude (stratified medium). An algorithm based on iterations on the sources is shown to be monotone and convergent. Numerical examples are given with highly varying absorption coefficient $\kappa$ and Rayleigh scattering as in the Earth atmosphere.  To study the effect of $\texttt{CO}_2$ in the atmosphere $\kappa$ is changed in the frequency ranges where  $\texttt{CO}_2$ is absorbing.
\end{abstract}

\end{frontmatter}

\tableofcontents

\section{Introduction}

Understanding and computing a polychromatic electromagnetic radiation crossing a medium with non constant refractive index and non constant absorption and scattering is a challenge for which the works of S. Chandrasekhar \cite{CHA} and G. Pomraning \cite{POM} are fundamental. Applications are numerous in astrophysics, nuclear engineering, atmospheric sciences \cite{DUF} and more recently image synthesis \cite{graphics}.

The medium is too large and the wavelengths too small to use Maxwell's equations, so the  Radiative Transfer Equation (RTE) is used. It is an integro partial differential system in 6 dimensions, 3 for space, 2 for ray directions, 1 for frequencies; so it is a serious numerical challenge. With his coauthors the present author proposed to analyze the coupling of RTE with the temperature equation to understand the effect of the numerous physical parameters on the temperature in the Earth atmosphere. An algorithm based on iterations on the source was shown to be monotone and convergent for stratified media in \cite{FGOP3} and in 3D for a general topography in \cite{JCP}.
The method was generalized to polarized light by using the vector RTE (VRTE) for stratified media \cite{OP2023}.  Then it was further generalized to VRRTE (the second R is for refractive) for media with non-constant but smooth refractive index \cite{OP2024}.
\\
There are numerous methods to solve numerically the VRTE, based on Fourier or Chebyshev expansion \cite{dehann}, \cite{wang}, discrete ordinate \cite{weng}, lattice Boltzmann \cite{zhang}, Finite Elements \cite{tan}, etc.  They aim at giving a complete description of the Stokes vector as a function of spatial and ray directional variables. Here we are interested in the temperature and this does not require to solve the full VRTE, it is enough to compute the directional average of the Stokes vector.  Solving the VRTE coupled with the temperature equation has not received much attention, to our knowledge.
\\\\
In this article the method proposed in \cite{FGOP3} is generalized to VRRTE with discontinuous graded refractive index for which the Fresnel conditions are necessary to match the Stokes vector on both sides of the discontinuity.  We have used a formulation of Fresnel's conditions given by A. Garcia \cite{garcia} which is well adapted to the VRRTE.

While Fresnel's conditions are natural jump conservations for the Maxwell equations, they are not natural to the VRRTE.  Hence, we had to include them in the integral semi-analytic solution of the VRRTE rather than in the partial differential equations (see sections 5 and 6).

By coupling the VRRTE with the temperature equation the problem becomes nonlinear.  Iterations on the sources is a very simple idea in which the equations are solved with  given right hand side (the sources) and then the sources are updated with the new solution. The temperature equation is solved with Newton iterations as explained in \cite{FGOP3}.  We prove here that the sequences are monotonous and that the solution can be approached from above and below, at least when the jump in the refractive index is not too large.

The method is tested numerically on two set of cases, one in which the light comes from the Sun and the other where the (infrared) light comes from Earth.  These waves cross a medium which has a layer of large refractive index (like the sea) and the atmosphere above it which has a refractive index close to unity.

We intend to generalize the method to non stratified atmosphere  as in \cite{JCP}.

\section{Fundamental Equations}

Light in a medium $\Omega$  is an electromagnetic radiation satisfying Maxwell's equations.
The electric field ${\bf E}={\bf E_0}\exp(\ii({\bf k\cdot\vx} - \nu t))$ of a monochromatic plane wave of frequency $\nu$ propagating in direction ${\bf k}$,  is a solution to Maxwell equations which is suitable to describe the propagation of a ray for which $\nu$ is very large.

Such radiations are characterized either by ${\bf k}$ and $\vE_0$ or equivalently \cite{CHA}\cite{garcia} by their Stokes vectors $\vI=[I,Q,U,V]^T$,  made of  the irradiance $I$ and 3 functions $Q,U,V$ to define its state of polarization. 

The  radiation comes from the boundary but also from a the Planck law, a distributed source $\vF=[\kappa_a B_\nu(T),0,0,0]^T$  for an unpolarized-emitting black-body (for example due to the black-body radiation of air or water). It  is defined in terms of the rescaled  Planck function $B_\nu(T)=\nu^3(\e^\frac\nu T - 1)^{-1}$, and the rescaled temperature $T$. The range of frequencies of interest is $\nu\in(0.01,20)\times 10^{14}$, hence a scaling is applied, detailed in \cite{FGOP3}: $\nu$ is divided by $10^{14}$ and $T_K$ in Kelvin is divided by 4798:
$
T= 10^{-14}\frac kh T_K = \frac{T_K}{4798}, 
$
where $k$ and $h$ are the Boltzmann and Planck constants.
The parameter $\kappa_a$ is related to absorption and scattering (see \eqref{def} below),   
which, by the way, are quantum effects, not described by Maxwell's equations. 

Absorption and scattering are modeled by a system of integro-differential partial differential equations, known as VRRTE (short for Vector Radiative Refractive Transfer Equations )\cite{POM} p152, \cite{BEN}. 

With $\tilde\vI:=\vI/n^2$,  
\begin{align}\label{fundamental}
\frac nc\partial_t\tilde\vI &+ \vom\n_\vx \tilde\vI + \frac{\n_\vx n}{n}\cdot\n_\vom\tilde\vI+\kappa_\nu\tilde\vI 
= \int_\SS\Z_\nu(\vx,\vom':\vom)\tilde\vI\d\omega' + \tilde\vF_\nu,
\end{align}
for all $\nu\in\R^+$, $\vx\in\Omega,\vom\in\SS$, where $c$ is the speed of light, $n$ the refractive index of the medium ,  $\SS$ the unit sphere, $\kappa$ the absorption  and $\Z$ the phase scattering matrix for rays $\vom'$ scattered in direction $\vom$ for each frequency $\nu$.  
It is assumed that  $n$ depends continuously on position $\vx\in\Omega$ except on surfaces of discontinuities on which additional jump conditions will be applied (Fresnel's conditions); $\kappa$ depends  on $\vx$ and strongly on $\nu$. VRRTE \eqref{fundamental} does not hold at interfaces of strong discontinuities of $n$ where the transmission, reflection and refraction are subject to Fresnel's conditions. 

Because $c$ is very large, the term $\frac1c\partial_t\vI$ is neglected. 
The thermal conductivity is also small so that ``Thermal Equilibrium'' is assumed:
\begin{equation}
\label{thermal}
\n_\vx\cdot\int_{\R_+}\int_{\S_2}\tilde I\vom\d\omega\d\nu=0.
\end{equation}
\begin{notation}
On all variables, the tilde indicates a division by $n^2$. Arguments of functions are sometimes written as indices like $\kappa_\nu$ and $n_z$.
\end{notation}
Following \cite{Liu}, given a cartesian frame ${\bf i},{\bf j},{\bf k}$,  the third term on the left in \eqref{fundamental} is computed in polar coordinates, with 
\[
\vom := {\bf i}\sin\theta\cos\varphi+{\bf j}\sin\theta\sin\varphi+{\bf k}\cos\theta,
\quad{\bf s}_1:=-{\bf i}\sin\varphi+{\bf j}\cos\varphi,
\]
\begin{align*}
\n_\vx\log{n}\cdot\n_\vom\tilde\vI = \frac{1}{\sin \theta} \frac{\partial}{\partial \theta}\left\{\tilde\vI\cdot(\cos\theta \vom-\boldsymbol{k}) \cdot \nabla_\vx\log n\right\} 
+\frac{1}{\sin \theta} \frac{\partial}{\partial \varphi}\left\{\tilde\vI~\boldsymbol{s}_1 \cdot \nabla_\vx\log n\right\}.
\end{align*}
When $n$ does not depend on $x,y$ but only on $z$, and nothing depends on $\varphi$, it simplifies to
\begin{equation*}
\n_\vx\log{n}\cdot\n_\vom\tilde\vI = (\partial_z\log n)\cdot \partial_\mu\left\{(1-\mu^2)\tilde\vI\right\}
\quad\text{ where }\mu=\cos\theta.
\end{equation*}

\section{The Stratified Case}

For an atmosphere of thickness $Z$ over a flat ground, the spatial domain is $\Omega=\R^2\times(0,Z)$, but if  all variables are independent of $x,y$ it reduces to $(0,Z)$.
In that case, in \cite{CHA} p40-53, expressions for the phase matrix $\Z$ are given for Rayleigh and isotropic scattering for  
$[I,Q]^T$,
\begin{align*}&\ds
 \Z_R = \frac32\left[\begin{matrix}
2(1-\mu^2)(1-\mu'^2)+\mu^2\mu'^2 & \mu^2 \cr
\mu'^2 & 1\cr
\end{matrix}\right],
\quad
\Z_I = 
\frac12\left[\begin{matrix}
1&1 \cr
1&1\cr
\end{matrix}\right]
\end{align*}
For a given $\beta\in[0,1]$, we shall consider a combination of $\beta\Z_R$  (Rayleigh scattering) plus $(1-\beta)\Z_I$ (isotropic scatterings) as in \cite{CHA},\cite{POM},\cite{POM2}. 

The two other components of the Stokes vectors have autonomous equations,
\begin{align}\label{U}
\mu\p_z \tilde U + \partial_z\log n\cdot\partial_\mu\{(1-\mu^2)\tilde U\}
+\kappa\tilde U = 0, 
\\ \label{V}
\mu\p_z \tilde V + \partial_z\log n\cdot\partial_\mu\{(1-\mu^2)\tilde V\}
+\kappa\tilde V = \frac\mu{2}\int_{-1}^1 \mu'\tilde V(z,\mu')\d\mu'.
\end{align}
\begin{notation}
Denote the scattering coefficient $a_s\in[0,1)$, which, as $\kappa$, is a function of altitude $z$ and frequency $\nu$. Define  
\eq{\label{def}
\kappa_s =\kappa a_s, \qquad \kappa_a=\kappa-\kappa_s = \kappa (1-a_s).
}
\end{notation}
From \eqref{U},\eqref{V} we see that, if the light source at the boundary is unpolarized then  $U=V=0$ and the light can be described either by $I$ and $Q$ or two orthogonal components $I_l,I_r$, such that  $I=I_l+I_r$ and $Q=I_l-I_r$ (see  \cite{CHA}):
 \begin{equation}
\left\{\begin{aligned}\label{lllreq}&
 \mu\p_z  {\tilde I}_l + \partial_z\log n\cdot\partial_\mu\{(1-\mu^2)\tilde I_l\}+\kappa {\tilde I}_l 
 \cr&
 \hskip 2cm= \frac{3\beta\kappa_s}8 \int_{-1}^1([2(1-\mu'^2)(1-\mu^2)+\mu'^2\mu^2]{\tilde I}_l + \mu^2 {\tilde I}_r)\d\mu'
\cr&
\hskip 2cm + \frac{(1-\beta)\kappa_s}4\int_{-1}^1[{\tilde I}_l+ {\tilde I}_r]\d\mu'  +  \frac{\kappa_a}2 \tilde B_\nu(T(z)), 
 \cr&
 \mu\p_z  {\tilde I}_r + \partial_z\log n\cdot\partial_\mu\{(1-\mu^2)\tilde I_r\}+\kappa {\tilde I}_r =\frac{3\beta\kappa_s}8 \int_{-1}^1(\mu'^2 {\tilde I}_l + {\tilde I}_r)\d\mu' 
 \cr&
 \hskip 2cm  
+ \frac{(1-\beta)\kappa_s}4\int_{-1}^1[{\tilde I}_l+ {\tilde I}_r]\d\mu'  +  \frac{\kappa_a}2 \tilde B_\nu(T(z)), 
\end{aligned}\right.
\end{equation}

Using the above linear combination on \eqref{lllreq}, the system for $\tilde I$ and $\tilde Q$ is derived,
\begin{equation}
\left\{\begin{aligned}\label{lq}\ds &
\mu \p_z\tilde I + \partial_z\log n\cdot\partial_\mu\{(1-\mu^2)\tilde I\}+ \kappa\tilde I 
\cr&
\hskip 2cm=\kappa_a \tilde B_\nu + \frac{\kappa_s}2\int_{-1}^1 \tilde I\d\mu'
+ \frac{\beta\kappa_s}4 P_2(\mu)\int_{-1}^1 [P_2\tilde I-(1-P_2 )\tilde Q]\d\mu',
\cr&
\mu \p_z \tilde Q + \partial_z\log n\cdot\partial_\mu\{(1-\mu^2)\tilde Q\} + \kappa\tilde Q 
\cr&
\hskip 2cm= -\frac{\beta\kappa_s}4 (1-P_2(\mu))\int_{-1}^1 [P_2\tilde I-(1-P_2 )\tilde Q]\d\mu',
\end{aligned}\right.
\end{equation}
where  $P_2(\mu)=\tfrac12(3\mu^2-1)$.
The temperature $T(z)$ is linked to $I$ by \eqref{thermal} which, in the case of \eqref{lq} is as follows.
\begin{lemma}
Thermal equilibrium for \eqref{lllreq} or \eqref{lq} is
\eeqn{
   \int_{\R_+}\kappa_a\big[\tilde B_\nu(T)
 - \tfrac12\int_{-1}^1 \tilde I\d\mu\big]\d\nu=0.
}
\end{lemma}
\begin{proof}
Averaging in $\mu$ the first equation of \eqref{lq} leads to
\begin{align*}
\nabla_\vx\cdot\int_\SS \vom \tilde I = \partial_z(\frac12\int_{-1}^1\mu \tilde I\d\mu) &= 
-\frac12\partial_z\log n\cdot\int_{-1}^1\partial_\mu\{(1-\mu^2)\tilde I\}\d\mu
\cr&
- \frac12 \kappa\int_{-1}^1\tilde I \d\mu+\frac12\int_{-1}^1\kappa_a \tilde B_\nu\d\mu + \frac{\kappa_s}2\int_{-1}^1 \tilde I\d\mu',
\end{align*}
because $\int_{-1}^1P_2(\mu)\d\mu=0$. Now the first term on the right integrates to zero and $\kappa-\kappa_s=\kappa_a$.
\end{proof}
\begin{remark}
Note that if $n$ is discontinuous at $z=Y$ one expects $(1-\mu^2)[\tilde I,\tilde Q]^T$ constant in $\mu$ at $z=Y$. However the equations are not valid at $Y$ and if $n$ is constant before and after $Y$ with a jump at $Y$ it is not clear that a standard numerical method would see the term containing the Dirac mass $\partial_z\log n$.  On the other hand in the integral formulation that follows the characteristics change at $Y$ and so $\tilde \vI$ has a jump at $Y$ even without the Fresnel Conditions.
\end{remark}
\begin{orientation}
For the numerical simulations \eqref{lq} is more appropriate, but to derive energy estimates \eqref{lllreq} is better.  The differences are in the source terms and the boundary conditions but we can easily switch from one to the other. 
\end{orientation}

\section{A Stratified Medium with a Discontinuous Refractive Index}
Systems \eqref{lllreq} and \eqref{lq} are not valid across a discontinuity of $z\mapsto n(z)$, but the Fresnel Conditions give the needed jump conditions to patch the solutions.
Consider 3 parallel planes at $z=0$, $z=Y>0$  and $z=Z>Y$. The refractive index of the medium is $n^-$ when $z<Y$ and $n^+$ when $z>Y$. Denote, when the roots exist,
\begin{equation*}
 n _{\mp}=\frac{n_-}{n_+}, ~~  n _{\pm}=\frac{n_+}{n_-},
 \quad
 \eta(n) = \sqrt{1-n^2(1-\mu^2)},
 \quad 
 \mu_c(n)=\sqrt{1-\frac1{n^2}}.
\end{equation*}
%

Fresnel's refraction conditions are written in \cite{garcia} for $\vI$. Rewritten for $\tilde\vI$, they are,
\begin{align}\label{fresnel}
&\tilde\vI(Y^-,-\mu )=\mathbf{X}(n_{\mp}, \mu) \tilde\vI(Y^-, \mu )+\mathbf{Y}(n_{\mp}, \mu) \tilde\vI(Y^+,-\eta(n_{\mp}, \mu)) , \quad \mu\in(0,1), 
\cr&
\tilde\vI(Y^+, \mu )=\mathbf{X}(n_{\pm}, \mu) \tilde\vI(Y^+,-\mu )+\mathbf{Y}(n_{\pm}, \mu) \tilde\vI(Y^-, \eta(n_{\pm}, \mu) ), \quad \mu\in(0,1).~~ 
\end{align}
\begin{align*}
&\mathbf{X}(n, \mu)= \begin{cases}\mathbf{G}(n, \mu), & n \leq 1, \\
\mathbf{G}(n, \mu) H\left[\mu-\mu_c(n)\right]+\boldsymbol{\Gamma}(n, \mu)\left\{1-H\left[\mu-\mu_c(n)\right]\right\}, & n \geq 1,\end{cases}\\
&\mathbf{Y}(n, \mu)= \begin{cases}\mathbf{D}(n, \mu), & n \leq 1, \\ \mathbf{D}(n, \mu) H\left[\mu-\mu_c(n)\right], & n \geq 1.
\end{cases}
\end{align*}
Here $\tilde\vI\in\R^4$, $\vX,\vY$ are $4\times 4$ matrices given in terms of $H$, the Heaviside function  and 3 matrices ${\bf G, D, \Gamma}$, for which the non-zero terms  are ,
$$
\begin{gathered}
 \Gamma_{11}=\Gamma_{22}=1,\\
G_{11} =G_{22} =\frac{1}{2}\left\{\left[\frac{\mu-n \eta(n, \mu)}{\mu+n \eta(n, \mu)}\right]^2+\left[\frac{n \mu-\eta(n, \mu)}{n \mu+\eta(n, \mu)}\right]^2\right\} \\
G_{12} =G_{21} =\frac{1}{2}\left\{\left[\frac{\mu-n \eta(n, \mu)}{\mu+n \eta(n, \mu)}\right]^2-\left[\frac{n \mu-\eta(n, \mu)}{n \mu+\eta(n, \mu)}\right]^2\right\} \\
D_{11} =D_{22} =2 n \mu \eta(n, \mu)\left\{\frac{1}{[\mu+n \eta(n, \mu)]^2}+\frac{1}{[n \mu+\eta(n, \mu)]^2}\right\} \\
D_{12} =D_{21} =2 n \mu \eta(n, \mu)\left\{\frac{1}{[\mu+n \eta(n, \mu)]^2}-\frac{1}{[n \mu+\eta(n, \mu)]^2}\right\}
\end{gathered}
$$
\begin{align*}
\Gamma_{33}=\Gamma_{44} =1- \frac{2(1-\mu^2)^2}{1-(1+n^{-2})\mu^2},
\quad
\Gamma_{43}=-\Gamma_{34}=\frac{2\mu(1-\mu^2)(\mu_c^2-\mu^2)^\frac12}{1-(1+n^{-2})\mu^2}
\\
G_{33}=G_{44}=\left(\frac{\mu-n \eta(n, \mu)}{\mu+n \eta(n, \mu)}\right)\left(\frac{n\mu- \eta(n, \mu)}{n\mu+ \eta(n, \mu)}\right),
\\
D_{33}=D_{44}=\frac{4 n \mu \eta(n)}{(n \eta(n)+\mu)(n\mu+ \eta(n))}.
\end{align*}

\begin{remark}\label{rem1}
 Notice that
\begin{itemize}
\item System \eqref{fresnel} is compatible with a Stokes vector like $\tilde\vI=[I,Q,0,0]^T$: the last 2 components of $\tilde\vI$ on the left and right side of the equations can be zero. Therefore, when the polarization is with $U=V=0$, we can work with $\tilde\vI=[\tilde I,\tilde Q]^T$ and the $2\times 2$ matrices obtained from the left upper part of the full matrices.
\item  Notice that if $n\equiv 1$ then $\vX=0$ and $\vY={\bf 1}$.
\item Finally, notice  that the eigenvalues of the $2\times 2$ matrices $\vX$ and $\vY$ are real and less or equal to $1$.
\end{itemize}
\end{remark}
\begin{proof}
\[
\vX=\left(\begin{matrix}a+b & a-b \cr a-b & a+b \end{matrix}\right)
\quad%
\hbox{with $a=\frac{1}{2}\left[\frac{\mu-n \eta(n, \mu)}{\mu+n \eta(n, \mu)}\right]^2$,  $b=\frac12\left[\frac{n \mu-\eta(n, \mu)}{n \mu+\eta(n, \mu)}\right]^2$}
\]
The eigenvalues $\lambda$ are solutions of 
\[
\lambda^2 -2\lambda(a+b)+ 4ab=0\quad \Rightarrow~\lambda_1=2a,~~  \lambda_2=2b.
\]
It is similar for $\vY$ with $a=\frac{2 n \mu \eta(n, \mu)}{[\mu+n \eta(n, \mu)]^2}\le \frac 12$ and $b=\frac{2 n \mu \eta(n, \mu)}{[n\mu+ \eta(n, \mu)]^2}\le \frac 12$.
\end{proof}
\begin{remark}\label{rem2}
The Fresnel conditions written for $\tilde \vI=[\tilde I_l,\tilde I_r]^T$ have  $\vX$ and $\vY$ given by the same formulae but with the matrices changed to
\[
\vX':=\left(\begin{matrix} \vX_{11}+\vX_{12} &0\cr 0& \vX_{11}-\vX_{12} \end{matrix}\right), \quad \vY':=\left(\begin{matrix} \vY_{11}+\vY_{12} &0\cr 0&  \vY_{11}-\vY_{12} \end{matrix}\right).
\]
\end{remark}

\begin{notation}
From now on the tildes are dropped and $\vI$ etc are understood as $\tilde \vI$ etc.
\end{notation}

\section{Stratified VRRTE with Graded Index and Planar Discontinuity}
Consider the partition $(0,Z)=(0,Y]\cup[Y,Z)$. Assume that the refractive index is $n_z=n^-(z)$ in $(0,Y)$ and $n_z=n^+(z)$ in $(Y,Z)$. We assume that $\{n^+,n^-\}$ are smooth functions of  $z$.
The convective part of the vector radiative transfer equations for $\vI=[I,Q]^T$ is (tildes are dropped),
\begin{align}\label{vrte}
\mu\partial_z\vI +  \partial_z\log n\cdot\partial_\mu\{(1-\mu^2)\vI\} + \kappa(z)\vI = \vS,
\end{align}
with the source terms $\vS=[S_1(z,\mu),S_2(z,\mu)]^T$.
The characteristics of \eqref{vrte} divided by $\mu$ are defined by
\[
\dot \xi(s) = 1, \quad \dot\omega(s)=\partial_z\log n(z(s))\frac{1-\omega^2(s)}{\omega(s)}
\]
So the characteristic curve passing through $z$ and $\mu$ is
\[
z'\mapsto\omega(z')={\rm{sign}}(\mu)\sqrt{1-\left(\frac{n(z)}{n(z')}\right)^2(1-\mu^2)}.
\]
The solution of \eqref{vrte} is obtained by the method of characteristics, which can be adapted to the case of a discontinuity at $z=Y$,
\begin{align}\label{gensolpm}
&\vI(z,\mu)|_{z<Y} = \One_{\mu>0}\left[\e^{-\int_0^z\frac{\kappa(z')}{\omega(z')}d z'}\vI(0,\omega(0)) 
+ \int_0^z\frac{\e^{-\int^z_{z'} \frac{\kappa(z'')}{\omega(z'')}d z''}}{\omega(z')}{\vS(z',\omega(z'))}\d z'\right]
\cr&
+\One_{\mu<0}\left[
\e^{\int_z^Y\frac{\kappa(z')}{\omega(z')}d z'}\vI(Y^-,\omega(Y))
-\int_z^Y \frac{\e^{\int_{z}^{z'}\frac{\kappa(z'')}{\omega(z'')} d z'' }}{\omega(z')}{\vS(z',\omega(z'))}\d z'
\right]
\end{align}
\begin{align}\label{gensolpp}
&\vI(z,\mu)|_{z>Y} = \One_{\mu>0}\left[\e^{-\int_Y^z\frac{\kappa(z')}{\omega(z')}d z'}\vI(Y^+,\omega(Y)) 
+ \int_Y^z\frac{\e^{-\int^z_{z'} \frac{\kappa(z'')}{\omega(z'')}d z''}}{\omega(z')}{\vS(z',\omega(z')}\d z'\right]
\cr
&+\One_{\mu<0}\left[
\e^{\int_z^Z\frac{\kappa(z')}{\omega(z')}d z'}\vI(Z,\omega(Z))
-\int_z^Z \frac{\e^{\int_{z}^{z'}\frac{\kappa(z'')}{\omega(z'')} d z'' }}{\omega(z')}{\vS(z',-\omega(z'))}\d z'
\right]
\end{align}
\begin{notation}
~
Define, when possible,
\begin{equation*}
\fbox{$
\begin{array}{rcl}
&\eta(n):= \sqrt{1-n^2(1-\mu^2)}, \quad &
\ds \phi(z',z'')|_{z''\ge z'}= \exp\left\{{-\ds\int_{z'}^{z''}\frac{\kappa(y)}{\eta(\frac{n_z}{n_y})}d y}\right\},
\\
&\I(y)=\eta(\frac{n_z}{n_y})\vI(y,{\rm{sign}}(\mu)\eta(\frac{n_z}{n_y})),
\quad
&\S(y):=\vS(y,\eta(\frac{n_z}{n_y})) / \eta(\frac{n_z}{n_y}).
\end{array}
$}
\end{equation*}%

All are functions of $z$ and $\mu$ which are the reference point and direction to define the characteristic.
\end{notation}

\begin{lemma}
Assume that $\vS(z,-\mu)=\vS(z,\mu)$ for all $\mu$. If, for some function $\Delta(z,\mu)$,
\begin{equation}\label{deltaa}
 \I(Y^+)= \I(Y^-) +\Delta(z,\mu),
 \end{equation}
 then  the solution of \eqref{vrte} is
 \begin{align}\label{gensolppm}
\vI(z,\mu) &= \One_{\mu>0}\left[\phi(0,z)\I(0) 
+ \int_0^z\phi(z',z)\S(z')\d z'\right]
\cr&
+\One_{\mu<0}\left[
\phi(z,Z)\I(Z)
+\int_z^Z \phi(z,z')\S(z')\d z'
\right]
\cr&
+
[\One_{z>Y}\One_{\mu>0}\phi(Y,z)
-\One_{z<Y}\One_{\mu<0}\phi(z,Y)]\Delta(z,\mu).
\end{align}
\end{lemma}

\begin{proof}

With these notations \eqref{gensolpm} and \eqref{gensolpp} are
\begin{align}\label{gensolpms}
\vI(z,\mu)|_{z<Y} &= \One_{\mu>0}\left[\phi(0,z)\I(0) 
+ \int_0^z\phi(z',z){\S(z')}\d z'\right]
\cr&
+\One_{\mu<0}\left[
\phi(z,Y)\I(Y^-)
+\int_z^Y \phi(z,z'){\S(z')}\d z'
\right].
\end{align}
\begin{align}\label{gensolpps}
\vI(z,\mu)|_{z>Y} &= \One_{\mu>0}\left[\phi(Y,z)\I(Y^+) 
+ \int_Y^z\phi(z',z){\S(z')}\d z'\right]
\cr&
+\One_{\mu<0}\left[
\phi(z,Z)\I(Z)
+\int_z^Z \phi(z,z')\S(z')\d z'
\right].
\end{align}
Using \eqref{deltaa} in \eqref{gensolpms},
\begin{align*}
&\vI(z,\mu)|_{z<Y} = \One_{\mu>0}\left[\phi(0,z)\I(0) 
+ \int_0^z\phi(z',z)\S(z')\d z'\right]
\cr&
+\One_{\mu<0}\left[
\phi(z,Y)[(\I(Y^+)-\Delta(z,\mu)]
+\int_z^Y \phi(z,z')\S(z')\d z'
\right].
\end{align*}
Now by \eqref{gensolpps} used with $z=Y^+$, $\mu<0$,
\begin{align*}
\I(Y^+)|_{\mu<0} =
\phi(Y,Z)\I(Z)
+\int_Y^Z \phi(Y,z')\S(z')\d z'.
\end{align*}
Consequently,
\begin{align*}
&\vI(z,\mu)|_{z<Y} = \One_{\mu>0}\left[\phi(0,z)\I(0) 
+ \int_0^z\phi(z',z)\S(z')\d z'\right]
\cr&
+\One_{\mu<0}\left[
\phi(z,Y)[\phi(Y,Z)\I(Z)
+\int_Y^Z \phi(Y,z')\S(z')\d z'-\Delta(z,\mu)]
\right.\cr&\left.
+\int_z^Y \phi(z,z')\S(z')\d z'
\right] \qquad {\Rightarrow}
\cr
 &\vI(z,\mu)|_{z<Y} = \One_{\mu>0}\left[\phi(0,z)\I(0) 
+ \int_0^z\phi(z',z)\S(z')\d z'\right]
\cr&
~~~~~~~~~~~~+\One_{\mu<0}\left[ \phi(z,Z)\I(Z)
+\int_z^Z \phi(z,z')\S(z')\d z'
-\phi(z,Y)\Delta(z,\mu)
\right], 
\end{align*}
because $\phi(z,Y)\phi(Y,Z)=\phi(z,Z)$ and $\phi(z,Y)\phi(Y,z')=\phi(z,z')$ when $z<Y$ and $z'>Y$.
Now, if $z>Y$ we use \eqref{deltaa} in \eqref{gensolpps} 
\begin{align*}
&\vI(z,\mu)|_{z>Y} 
= \One_{\mu>0}\left[\phi(Y,z)[\I(Y^-) +\Delta(z,\mu)]
+ \int_Y^z\phi(z',z)\S(z')\d z'\right]
\cr&
+\One_{\mu<0}\left[
\phi(z,Z)\I(Z)
+\int_z^Z \phi(z,z')\S(z')\d z'
\right],
\end{align*}
and \eqref{gensolpms} at $z=Y$, with $\mu>0$,
\begin{align*}
\I(Y^-)|_{\mu>0}  = \phi(0,Y)\I(0) 
+ \int_0^Y\phi(z',Y,\mu)\S(z')\d z'.
\end{align*}
It shows that
\begin{align*}
&\vI(z,\mu)|_{z>Y} = \One_{\mu>0}\left[\phi(Y,z)[
\phi(0,Y)\I(0) 
 + \int_0^Y\phi(z',Y)\S(z')\d z'
 +\Delta(z,\mu)]
 \right.\cr&\left.
+ \int_Y^z\phi(z',z)\S(z')\d z'\right]
+\One_{\mu<0}\left[
\phi(z,Z)\I(Z)
+\int_z^Z \phi(z,z')\S(z')\d z'
\right]\quad {\Rightarrow}
\cr&
\vI(z,\mu)|_{z>Y}=\One_{\mu>0}\left[\phi(0,z)\I(0) 
 + \int_0^z\phi(z',z)\S(z')\d z'
+\phi(Y,z)\Delta(z,\mu)
\right] 
\cr&
~~~~~~~~~ +\One_{\mu<0}\left[
\phi(z,Z)\I(Z)
+\int_z^Z \phi(z,z')\S(z')\d z'
\right].
\end{align*}
Therefore, the additional term due to the discontinuity is
\[
[\One_{z>Y}\One_{\mu>0}\phi(Y,z)
-\One_{z<Y}\One_{\mu<0}\phi(z,Y)]\Delta(z,\mu).
\]
\end{proof}

\section{Application to Fresnel's Conditions}
\begin{notation}
Define, when possible,
\begin{equation*}
\fbox{$
\begin{array}{rcl}
&\eta_\pm(n):=(1-n_\pm^2(1-\eta(n)^2)^\frac12   ,
 \quad &
 \phi_\pm(z',z''):=\exp\left\{\ds-\int_{z'}^{z''}\frac{\kappa(y)}{\eta_\pm(\frac{n_z}{n_y})}\d y\right\}
\\
&\I_\pm(y)=\vI(y,{\rm{sign}}(\mu)\eta_\pm(\frac{n_z}{n_y}))
\quad
&\S_\pm(y)= \vS(y,\eta_\pm(\frac{n_z}{n_y}))/\eta_\pm(\frac{n_z}{n_y})
\end{array}
$}
\end{equation*}%
and similarly with $\mp$ for $\eta_\mp$ etc.  Let
\[
\vX_\pm:=\vX(n_\pm,|\mu|),~~\vX_\mp:=\vX(n_\mp,|\mu|)
\text{ and similarly with $\vY$}.
\]
\end{notation}
\subsection{Computation of $\Delta(z,\mu)$}
 From \eqref{gensolpps} and \eqref{gensolpms} we obtain
 \begin{align*}&
 \I(Y^+)|_{\mu<0}= \phi(Y,Z)\I(Z) +\int_Y^Z\phi(Y,z')\S(z')\d z',
 \cr&
 \I(Y^-)|_{\mu>0}= \phi(0,Y)\I(0) + \int_0^Y\phi(z',Y)\S(z')\d z'.
\end{align*}

And \eqref{gensolpms} and \eqref{gensolpps}  plugged in the Fresnel conditions \eqref{fresnel} yield,
 \begin{align*}
 \I(Y^-)|_{\mu<0}&=  \vX_\mp \left[\phi(0,Y)\I(0,-\mu)+\int_0^Y\phi(z',Y)\S(z')\d z'\right]
 \cr&
 + \vY_\mp \left[ \phi_\mp(Y,Z,)\I_\mp(Z)
  +\int_Y^Z \phi_\mp(Y,z')\S_\mp(z')\d z'\right],
\cr
\I(Y^+)|_{\mu>0} &=  \vX_\pm \left[ \phi(Y,Z)\I(Z,-\mu)+\int_Y^Z\phi(Y,z')\S(z')\d z'\right]
 \cr&
 + \vY_\pm \left[\phi_\pm(0,Y)\I_\pm(0)
  +\int_0^Y \phi_\pm(z',Y)\S_\pm(z')\d z'\right].
\end{align*}
We have added the dependency on $\mu$ on some of the functions because of the change of sign requested by the Fresnel conditions.

Therefore,
 \begin{align}\label{delta}
 \Delta(z,\mu)|_{\mu>0}&=\I(Y^+)|_{\mu>0}-\I(Y^-)|_{\mu>0} 
 \cr&
 = \vX_\pm \left[ \phi(Y,Z)\I(Z,-\mu)+\int_Y^Z\phi(Y,z')\S(z')\d z'\right]
 \cr&
 + \vY_\pm \left[\phi_\pm(0,Y)\I_\pm(0)
  +\int_0^Y \phi_\pm(z',Y)\S_\pm(z')\d z'\right] 
  \cr&
-\phi(0,Y)\I(0) - \int_0^Y\phi(z',Y)\S(z')\d z',
\cr
 \Delta(z,\mu)|_{\mu<0}&=\I(Y^+)|_{\mu<0}-\I(Y^-)|_{\mu<0} 
 \cr&
= \phi(Y,Z)\I(Z) +\int_Y^Z\phi(Y,z')\S(z')\d z'
\cr&
 -\vX_\mp \left[\phi(0,Y)\I(0,-\mu)+\int_0^Y\phi(z',Y)\S(z')\d z'\right]
 \cr&
 - \vY_\mp \left[ \phi_\mp(Y,Z,)\I_\mp(Z)
  +\int_Y^Z \phi_\mp(Y,z')\S_\mp(z')\d z'\right].
 \end{align}
\section{Iterations on the sources, the numerical scheme}
Consider
 \begin{align}\label{lqn}
{\vI}^{n+1}(z,\mu) &= \One_{\mu>0}\left[\phi(0,z)\I(0) 
+ \int_0^z\phi(z',z){\S}^n(z')\d z'\right]
\cr&
+\One_{\mu<0}\left[
\phi(z,Z)\I(Z)
+\int_z^Z \phi(z,z'){\S}^n(z')\d z'
\right]
\cr&
+
[\One_{z>Y}\One_{\mu>0}\phi(Y,z)
-\One_{z<Y}\One_{\mu<0}\phi(z,Y)]\Delta^n(z,\mu),
\end{align}
where $\Delta^n$ is given in terms of ${\S}^n$ by \eqref{delta}. 

\subsection{Implementation}
To implement the iterative algorithm,  the only functions needed are $\int_{-1}^1\vI(z,\mu)\mu^k\d\mu, ~k=0,2$ and that means that we need to  compute (see below) 
\[
\int_{0}^1\mu^k\phi(Y,z)\Delta(z,\mu)\d\mu,~z>Y\hbox{ and }\int_{-1}^0\mu^k\phi(z,Y)\Delta(z,\mu)\d\mu,~z<Y.
\]
  Consequently, with $\vS=\vS_0+\mu^2\vS_2$,
\begin{align}\label{zalpha}
&\int_{-1}^1\mu^k[\Delta(z,\mu)[\phi(Y,z)\One_{z>Y}\One_{\mu>0}-\phi(z,Y)\One_{z<Y}\One_{\mu<0}]
\d\mu
\cr&
= {\bm\alpha}^k(z) + \sum_{i=0,2}\int_0^Z \vZ^{k,{i-1}}(z,z') \vS_i(z')\d z',
\end{align}
To speedup the computations, ${\bm\alpha}$ and $\vZ$ are tabulated before hand.

The $2\times 2$ matrices $\vZ^{k,i}$, $i,k=0,2$, are
\begin{align}\label{ZZ}
\vZ^{k,i}(z,z') 
		&:= \int_{0}^1\mu^k \left(
\One_{z>Y}\phi(Y,z)\left\{\vX_\pm {\phi(Y,z')}\eta^i(\frac{n_z}{n_{z'}})\One_{z'>Y}
\right.\right.\cr&\left.\left.
+[\vY_\pm {\phi_\pm(z',Y)}{\eta_\pm^i(\frac{n_z}{n_{z'}})}
-\One\phi(z',Y)\eta^i(\frac{n_z}{n_{z'}})]\One_{z'<Y}
\right\}
\right.\cr&\left.
 +\One_{z<Y}  \phi(z,Y)
 \left\{
 	\vX_\mp {\phi(z',Y)}\eta^i(\frac{n_z}{n_{z'}})\One_{z'<Y}
	\right.\right.\cr&\left.\left. 
		+[\vY_\mp \phi_\mp(Y,z')\eta_\mp^i(\frac{n_z}{n_{z'}})
		-\One\phi(Y,z')\eta^i(\frac{n_z}{n_{z'}})]
	\One_{z'>Y}
  \right\}
  \right)\d\mu
\end{align}
and the vectors, $k=0,2$,
\begin{align*}
{\bm\alpha}^k(z) =
 \One_{z>Y} \int_0^1\mu^k 
 \phi(Y,z)\left[
 \vX_\pm  \phi(Y,Z)\I(Z,-\mu)
  + \vY_\pm \phi_\pm(0,Y)\I_\pm(0)
 \right. \\ \left.
  - \phi(0,Y)\I(0)
  \right]\d\mu
 \\
 + 
   \One_{z<Y}\int_{0}^1\mu^k \phi(z,Y)\left[
 \vX_\mp \phi(0,Y)\I(0)
    +\vY_\mp  \phi_\mp(Y,Z)\I_\mp(Z) 
    \right. \\ \left.
-\phi(Y,Z)\I(Z,-\mu)\right]\d\mu
 \end{align*}
\begin{remark}
Whenever feasible it is computationally advantageous to separate in ${\bm\alpha}^k$ the part containing $\nu$ from the one containing $z$. 
\end{remark}
For our purpose  
\begin{align*}&
\vI(0,\mu)=[\mu c_E B_\nu(T_E),0]^T, 
\quad
\vI(Z,-\mu)=[\mu c_S B_\nu(T_S),0]^T
\cr&
~~\Rightarrow~~
{\bm\alpha}^k(z) = [\alpha^k_E(z) B_\nu(T_E) + \alpha^k_S(z) B_\nu(T_S),0]^T
\end{align*}
with
\begin{align*}
\alpha^k_E(z)& =  
c_E \int_0^1\mu^k\left[ 
 \One_{z>Y}\phi(Y,z)[{\vY_\pm}_{11} \phi_\pm(0,Y)\eta_\pm(\frac{n_z}{n_0})
  - \phi(0,Y)\eta(\frac{n_z}{n_0})]
\right. \\ & \left. 
~~~~~~~~~~~~ 
+ \One_{z<Y}\phi(z,Y)
 {\vX_\mp}_{11} \phi(0,Y)\eta(\frac{n_z}{n_{0}})\right]\d\mu
    \cr
\alpha^k_S(z) &=
c_S \int_0^1\mu^k 
\left[
\One_{z>Y} \phi(Y,z)
 {\vX_\pm }_{11} \phi(Y,Z)\eta(\frac{n_z}{n_{Z}})
\right. 
\\ &
 \left.
 ~~~~~~~~~~ ~~ 
 +\One_{z<Y}\phi(z,Y)
    [{\vY_\mp}_{11}  \phi_\mp(Y,Z)\eta_\mp(\frac{n_z}{n_{Z}})
  -  \phi(Y,Z)\eta(\frac{n_z}{n_{Z}})
  ]\right]\d\mu
\end{align*}
\section{Integral representation of the Problem}
\subsection{System $[I,Q]^T$}
With the main purpose of computing the temperature, let us denote
\begin{equation*}
  J_k(z) = \tfrac12\int_{-1}^1 \mu^k  I\d\mu,
\quad
  K_k(z) = \tfrac12\int_{-1}^1 \mu^k  Q\d\mu .
\quad k=0,2.
\end{equation*}
Consider system \eqref{lq} for the irradiance $I$ and the polarization $Q$,
\begin{align*}\ds 
\mu \p_z   I +\partial_z\log n\cdot\partial_\mu\{(1-\mu^2)  I\}&+ \kappa   I =\kappa_a   B_\nu + \kappa_s   J_0&
\cr &
+ \frac{\beta\kappa_s}4 P_2(\mu)(3   J_2 -   J_0 
-3   K_0 + 3   K_2),
\\ 
\mu \p_z   Q + \partial_z\log n\cdot \partial_\mu\{(1-\mu^2)  Q\}&+ \kappa   Q = 
\cr &-\frac{\beta\kappa_s}4 (1-P_2(\mu))(3  J_2-  J_0 - 3   K_0 +3   K_2),
\end{align*}
where  $P_2(\mu)=\tfrac12(3\mu^2-1)$.   
Hence
\begin{align}\label{SS}
\vS(z,\mu)=& [S_1,S_2]^T=[S^0_1+\mu^2 S_1^2,
 S_2^0+\mu^2 {S_2}^2]^T\hbox{ with }
\cr S_1= &\kappa_a   B_\nu + \kappa_s   J_0
+ \frac{\beta\kappa_s}4 P_2(\mu)(3   J_2 -   J_0 
-3   K_0 + 3   K_2),
\cr
 S_2=& -\frac{\beta\kappa_s}4 (1-P_2(\mu))(3  J_2-  J_0 - 3   K_0 +3   K_2),
 ,\quad \Rightarrow
 \cr
 &S_1^0=\kappa_a   B_\nu + \kappa_s   J_0- \frac13 H,
 \quad
 S_1^2= H,
  \quad
 S_2^0=\frac13 H,
 \quad
 S_2^2=-H,
 \cr &
 \text{with } 
 H=\frac{9\beta\kappa_s}8(   J_2 -   \frac13J_0 -   K_0 +    K_2).
\end{align}
Define
\[
\psi(z,z')=\One_{z>z'}\phi(z',z)+\One_{z<z'}\phi(z,z').
\]
Let \eqref{gensolppm} be multiplied  by $\mu^k$ and integrated. 
Then, by \eqref{zalpha},
\begin{align}\label{J0}\ds
 J_k(z) =  
\frac12\int_0^Z\left[\left(\int_0^1\mu^k\psi(z,z')\d\mu+\vZ^{k,0}_{11}(z,z')\right)S_1^0(z')
\right.\cr\left.
+\vZ^{k,0}_{12}(z,z')S_2^0(z')\right]\d z'
\cr
+\frac12\int_0^Z\left[\left(\int_0^1\mu^k\psi(z,z')\eta^2(\frac{n_z}{n_z'})\d\mu+\vZ^{k,2}_{11}(z,z')\right)S_1^2(z')
\right.\cr\left.
+\vZ^{k,2}_{12}(z,z')S_2^2(z')\right]\d z'
+\frac12 { \alpha_1^k}(z)
\end{align}
Similarly (recall that $  Q$ is zero at $z=0,\mu>0$ and $z=Z,\mu<0$),
\begin{align*}
  K_k(z) =
\frac12\int_0^Z\left[\left(\int_0^1\mu^k\psi(z,z')\d\mu+\vZ_{22}^{k,0}\right) S_2^0(z')+\vZ^{k,0}_{21}S_1^0(z'))
\right.\cr\left.
+\left(\int_0^1\mu^k\psi(z,z')\eta^2(\frac{n_z}{n_z'})\d\mu + \vZ_{22}^{k,2}\right)S^2_2(z')
\right.\cr\left.
+\vZ^{k,2}_{21}S_1^2(z'))\right]\d z' +\frac12\alpha_2^k(z).
\end{align*}

\subsection{System $[I_l,I_r]^T$}
Now define
\begin{equation}\label{JK}
  J'_k(z) = \tfrac12\int_{-1}^1 \mu^k  I_l\d\mu,
\quad
  K'_k(z) = \tfrac12\int_{-1}^1 \mu^k  I_r\d\mu .
\quad k=0,2.
\end{equation}
and consider \eqref{lllreq} 
\begin{equation}
\left\{\begin{aligned}\label{lllreq2}&
 \mu\p_z  { I}_l + \partial_z\log n\cdot\partial_\mu\{(1-\mu^2) I_l\}+\kappa { I}_l 
 \cr&
 \hskip 2cm= \frac{3\beta\kappa_s}8[2(1-\mu^2)(J'_0-J'_2) + J'_2\mu^2 + \mu^2 K'_0]
 \cr&
 \hskip 2cm +\frac{(1-\beta)\kappa_s}4[J'_0+K'_0]  +  \frac{\kappa_a}2 \tilde B_\nu(T(z)), 
 \cr&
 \mu\p_z  { I}_r + \partial_z\log n\cdot\partial_\mu\{(1-\mu^2) I_r\}+\kappa { I}_r =\frac{3\beta\kappa_s}8 [J'_2+K'_0] 
 \cr&
 \hskip 2cm
 +\frac{(1-\beta)\kappa_s}4[J'_0+K'_0] +  \frac{\kappa_a}2 \tilde B_\nu(T(z)).
\end{aligned}\right.
\end{equation}
With the same notation as for $[I,Q]^T$,
\begin{align}\label{sourcep}&
\vS'=[S'_1,S'_2]^T= [{S'}_1^0+\mu^2{S'}_1^2,{S'}_2^0+\mu^2{S'}_2^2]^T,
\cr&
{S'}_1^0 = \frac{3\beta\kappa_s}8[2J'_0-2J'_2]+ \frac{(1-\beta)\kappa_s}4[J'_0+K'_0]
+\frac{\kappa_a}2 \tilde B_\nu(T(z)),
\cr&
{S'}_1^2 = \frac{3\beta\kappa_s}8[3J'_2-2J'_0  + K'_0],
\quad
{S'}_2^2=0,
\cr&
{S'}_2^0 = \frac{3\beta\kappa_s}8[J'_2+K'_0 ] + \frac{(1-\beta)\kappa_s}4[J'_0+K'_0]
+ \frac{\kappa_a}2 \tilde B_\nu(T(z)).
\end{align}
as for $[I,Q]^T$, we have
\begin{align}\label{JKp}\ds
 J'_k(z) &=  
\frac12\int_0^Z\left(\int_0^1\mu^k\psi(z,z')\d\mu+{\vZ'}^{k,0}_{11}(z,z')\right){S'}_1^0(z')
\d z'
+\frac12 { {\alpha'}_1^k}(z)
\cr&
+\frac12\int_0^Z\left(\int_0^1\mu^k\psi(z,z')\eta^2(\frac{n_z}{n_z'})\d\mu+{\vZ'}^{k,2}_{11}(z,z')\right){S'}_1^2(z')
\d z',
\cr
  {K'}_k(z) &=
\frac12\int_0^Z\left(\int_0^1\mu^k\psi(z,z')\d\mu+{\vZ'}_{22}^{k,0}\right) S_2^0(z')\d z
 +\frac12{\alpha'}_2^k(z)
 \cr&
+\frac12\int_0^Z\left(\int_0^1\mu^k\psi(z,z')\eta^2(\frac{n_z}{n_z'})\d\mu 
+ {\vZ'}_{22}^{k,2}\right){S'}^2_2(z')
\d z'.
\end{align}

The same expressions \eqref{zalpha},\eqref{ZZ} hold for ${\bm\alpha'}$ and $\vZ'$ with $\vX'$ and $\vY'$ instead of $\vX,\vY$ as in Remark \ref{rem2}.

\begin{remark}
Note that 
\[
2J'_0-2 J'_2=\int_{-1}^1(1-{\mu'}^2 )I_l\d\mu'\ge 0,
\quad 
3J'_0-2 J'_2=\int_{-1}^1(3-2{\mu'}^2 )I_l\d\mu'\ge 0.
\]
Therefore $\vS'$ is always non negative.
\end{remark}

\section{Convergence of the Iterations on the Sources}
\subsection{Algorithm}
Just as before for $\vI=[I,Q]^T$, the iterations on $\vI'=[I_l,I_r]^T$ are 
 \begin{align}\label{lllr}
{\vI'}^{n+1}(z,\mu) &= \One_{\mu>0}\left[\phi(0,z)\I(0) 
+ \int_0^z\phi(z',z){\S'}^n(z')\d z'\right]
\cr&
+\One_{\mu<0}\left[
\phi(z,Z)\I(Z)
+\int_z^Z \phi(z,z'){\S'}^n(z')\d z'
\right]
\cr&
+
[\One_{z>Y}\One_{\mu>0}\phi(Y,z)
-\One_{z<Y}\One_{\mu<0}\phi(z,Y)]\Delta^n(z,\mu),
\end{align}
where $\Delta^n$ is given in terms of ${\S'}^n$ by \eqref{delta}. 
To update $\S'(z')={\vS'(z')}/\eta(\frac{n_z}{n_{z'}})$ we use first \eqref{JK} with ${\vI'}^{n+1}$ to compute ${J'}^{n+1}_k,{K'}^{n+1}_k$, then compute $T^{n+1}$ by solving
\begin{align}\label{getT}
\int_{\R_+}\kappa_a   B_\nu(T^{n+1})\d\nu &= \int_{\R_+}\kappa_a ({J'_0}^{n+1}+{K'_0}^{n+1})\d\nu,~~ \forall z\in(0,Z).
\end{align}
and finally obtain ${\vS'}^{n+1}$ by \eqref{sourcep}.

Note that a Newton method can be used to solve \eqref{getT}  and convergence is implied by the strict positivity of  $T\mapsto \partial_T B_\nu(T)$ and the boundedness and continuity of $\partial_{TT} B_\nu(T)$ in any interval $[T_m,T_M]$ containing the solution. However, as usual, one must not start too far from the solution and $\nu\in(\nu_m,\nu_M)$, $\nu_m>0$.
\\\\
\textbf{Important Observation}
\\
{\it Computing \eqref{JK} with ${\vI'}^{n+1}$  is equivalent to computing \eqref{JKp} with ${\vS'}^n$.  The second is very much cheaper numerically.  If, at some point $z,\mu$, $\vI'$ is desired then \eqref{lllr} is used but once the iteration process has converged.}

\subsection{Monotony}
In earlier studies, like \cite{FGOP3}, on simpler systems, convergence was shown by using the monotony of operators. Here too the same arguments are used, not on $[I,Q]^T$ but on $\vI'=[I_l,I_r]^T$.
If $z<Y<z'$ then $\phi(z,Y)\phi(Y,z')=\phi(z,z')$, so, \eqref{delta} leads to
\begin{align*}&
(\Delta^n(z,\mu)-\Delta^{n-1}(z,\mu))[\phi(Y,z)\One_{z>Y}\One_{\mu>0}-\phi(z,Y)\One_{z<Y}\One_{\mu<0}]
\cr&
= -\One_{z>Y}\One_{\mu>0}\int_0^Y\phi(z',z)({\S'}^n(z')-{\S'}^{n-1}(z'))\d z' 
\cr&
 -\One_{z<Y}\One_{\mu<0}\int_Y^Z\phi(z,z')({\S'}^n(z')-{\S'}^{n-1}(z'))\d z' 
\cr&
 + \int_0^Z{\cal L}(z,z',Y,\mu)[{\S'}^n(z')-{\S'}^{n-1}(z')]\d z',
\end{align*}
where ${\cal L}(z,z',Y,\mu)$, the terms linear in $\S$ containing $\vX'$ and $\vY'$, are non negative.
 Consequently,
\begin{align*}
{\vI'}^{n+1}(z,\mu) - {\vI'}^{n}(z,\mu) 
 &=
 \One_{\mu>0}
\left\{ \int_0^z{\phi(z',z)}({\S'}^n-{\S'}^{n-1})\d z'
\right.\cr&\left.
~~~~ -\One_{z>Y}\int_0^Y\phi(z',z)({\S'}^n-{\S'}^{n-1})\d z' 
\right\}\cr&
+\One_{\mu<0}
\left\{ \int_z^Z {\phi(z,z')}({\S'}^n-{\S'}^{n-1})\d z'
\right.\cr&\left.
 ~~~~ -\One_{z<Y}\int_Y^Z\phi(z,z')({\S'}^n-{\S'}^{n-1})\d z' 
\right\}\cr&
 + \int_0^Z{\cal L}(z,z',Y,\mu)[{\S'}^n-{\S'}^{n-1}]\d z'.
\end{align*}
Positivity of all multipliers of ${\S'}^n-{\S'}^{n-1}$ is now obvious because
\begin{align*}
\int_0^z{\phi(z',z)}({\S'}^n-{\S'}^{n-1})\d z'
-\One_{z>Y}\int_0^Y\phi(z',z)({\S'}^n-{\S'}^{n-1})\d z'
\\ = \int_0^z{\phi(z',z)}({\S'}^n-{\S'}^{n-1})\d z' \text{ if $z<Y$}
\\ = \int_Y^z{\phi(z',z)}({\S'}^n-{\S'}^{n-1})\d z' \text{ if $z>Y$},
\\
 \int_z^Z {\phi(z,z')}({\S'}^n-{\S'}^{n-1})\d z'
 -\One_{z<Y}\int_Y^Z\phi(z,z')({\S'}^n-{\S'}^{n-1})\d z'
\\ = \int_z^Z{\phi(z,z')}({\S'}^n-{\S'}^{n-1})\d z' \text{ if $z>Y$}
\\ = \int_z^Y{\phi(z,z')}({\S'}^n-{\S'}^{n-1})\d z' \text{ if $z<Y$} .
\end{align*}

Consequently,  ${\vS'}^n\ge {\vS'}^{n-1}$  implies  ${\S'}^n\ge {\S'}^{n-1}$ and hence ${\vI'}^{n+1}\ge {\vI'}^n$.  In turn, then, ${J'}_k^{n+1}\ge {J'}_k^n$ and ${K'}_k^{n+1}\ge {K'}_k^n$.

  Finally, the temperature equation implies 
 \eeqn{&
 \int_{\R_+}\kappa_a   B_\nu(T^{n+1})\d\nu &= \int_{\R_+}\kappa_a ({J'}_0^{n+1}+({K'}_0^{n+1})\d\nu 
\cr&&
\ge 
 \int_{\R_+}\kappa_a ({J'}^{n}_0+{K'}^{n}_0)\d\nu 
 =\int_{\R_+}\kappa_a   B_\nu(T'^{n})\d\nu.
}
which implies that  $T^{n+1}\ge T'^{n}$ because $T\mapsto B_\nu(T)$ is monotone increasing.

In summary, it shows that 
 \[
 T^n\ge T^{n-1},~ {  I}^n_{l,r}\ge {  I}_{i,r}^{n-1}~~ \Rightarrow ~~T^{n+1}\ge T^n,~ {  I}^{n+1}_{l,r}\ge {  I}_{i,r}^{n}.
 \]
  To start the iterations appropriately, simply set $T^0=0$, ${  I}^0_{l,r}=0$, then by the positivity of the coefficients ${  I}^1_{l,r}\ge 0$ and $T^1\ge 0$.

 The same argument works with  decreasing sequences,  
  \[
 T^n\le T^{n-1},~ {I}^n_{l,r}\le{I}_{i,r}^{n-1}~~ \Rightarrow ~~T^{n+1}\le T^n,~ {I}^{n+1}_{l,r}\le{I}_{i,r}^{n}.
 \]
Note that the decreasing  $\mu$ integrals of $\{\vI^n\}_n$ are bounded by zero. Consequently, convergence holds and a solution to the system $\{\vI',T\}$ exists.

But starting a decreasing sequence with $T^1\le T^0$, ${I}^1_{l,r}\le{I}^0_{l,r}$ is not so simple. In practice it seems that $T^{-1}=1,~I^{-1}_{l,r}=0$ works.

\begin{proposition}
If the solution $T^*,I^*_{l,r}$ exists,  then it can be reached numerically from above or below by iterations \eqref{lllr} and these are monotone increasing if initiated with $T^0=I^0_{l,r}=0$  and decreasing if initiated by $T^1\ge T^0$, and $ I^0_{l,r}$ such that  $I^0_{l,r}\ge I^1_{l,r}$  Furthermore,
if there is an initial guess for a decreasing sequence can be found, then convergence is implied   and a solution exists to the VRRTE system with Fresnel Conditions.
\end{proposition}

\section{Numerical Results}
The computer program is written in $C++$ and opensource. The absorption is  $\kappa=\rho(z)\overline\kappa(\nu)$ where $\rho$ is the density; $\overline\kappa_\nu$  is either constant and equal to 0.5 or taken from the Gemini experiment \footnote{
\url{www.gemini.edu/observing/telescopes-and-sites/sites\#Transmission}
}
as shown in figure \ref{kappafig}.  The presence of \texttt{CO}$_2$ in the atmosphere changes $\overline\kappa_\nu$ into  $\overline\kappa^1_\nu=\min\{(1.2,1.8\overline\kappa(\nu)$ in the range
\[
\nu\in[\frac3{18},\frac3{14}]\cup[\frac35,\frac32].
\]
This is shown also in figure \ref{kappafig}.

We investigated 2  cases:
\begin{itemize}
\item
Case 1: Visible light coming from the Sun through the top of the troposphere at $z=Z=10km$ and escaping freely at $z=0$, i.e.
\[
I(0,\mu)=0, \quad I(Z,-\mu)=c_S B_\nu(T_S)\mu, \mu>0.
\]
\item
Case 2: Infrared light coming from Earth and escaping freely at $z=Z$:  
\[
I(0,\mu)=c_E B_\nu(T_E)\mu,~I(Z,-\mu)=0, \mu>0.
\]
In both cases there is a change of refractive index at $z=Y=Z/2$.
According to \cite{measureN} the variation of the refractive index in the atmosphere due to clouds is quite small $\sim 0.003$. To enhance the effect we use 3 times this value.  
\item
Case 3: We also computed an atmosphere above an ocean 1000m deep with a change of index from 1 in water to 0.7 in air  and a change of density from 10 in water to 1/10 in air. In reality it should be 1/100, but then the exponentials should be evaluated differently. Infrared is coming from $z=0$ as in Case 1 or sunlight is coming from $Z$ as in Case 1.
\end{itemize}
For all tests the following is used:
\begin{itemize}
\item $n(z)= 1+\epsilon 
\One_{z>Y}$,
\item $a_s=a_1\One_{z\in(z_1,z_2)} + a_2\One_{z>z_2}\One_{\nu\in(\nu_1,\nu_2)}\left(\frac\nu{\nu_2}\right)^4$,
\item $\epsilon=-0.01$ or $-0.3$, $a_1=0.7,~a_2=0.3$, $z_1=0.4,~z_2=0.8$, $\nu_1=0.6,~\nu_2=1.5$.
\item $c_S = 2\cdot 10^{-5}, T_S = 5700/4798, c_E = 2.5, T_E =300/4798$. 
\end{itemize}

Functions $\int_0^1\phi(z,z',\mu)\mu^k\d\mu,~k=0,2$, $\vZ$ and ${\bm\alpha}$ are tabulated for 60 values of $\overline\kappa\in(0.01,1.2)$ to speed up the runtime which is around 10 seconds  on a high end Apple Macbook when $(0,Z)$ is discretized with 60 points and $(0,1)$ with 100 intervals.

The monotony of the iterative process is displayed in figure \ref{convergefig}.  It is clear that by starting below (resp. above) the solution - here the values of the temperature at $z=300$m - are increasing (resp. decreasing). Note that 15 iterations are sufficient to obtain a 3-digits precision.

\subsection{Part I: Comparison of the temperature jumps with and without Fresnel Conditions}

To study the effect of $n$ on a simple case we ran the program with $\kappa=0.5$, $n$ as above, with $\epsilon=0.01$ and the data of Case 2. The temperatures are shown in figure \ref{tempeK1} in red. The computations are done with and without Fresnel conditions at $z=Y$.

A similar computation is done for Case 1 with $\kappa=0.5$, $\epsilon=-0.3$ with and without Fresnel conditions.  The temperatures are displayed in blue in  figure \ref{tempeK1}.  The average ligh and polarization intensities are shown in figure \ref{lightI1}.
As expected the temperature jumps with and without Fresnel Conditions are different.

\subsection{Part II: Effect of \texttt{CO}$_2$ when $n$ has a Discontinuity}

Case 2 was run with $\epsilon=-0.3$ with $\kappa_\nu$ read from the Gemini website and shown in figure \ref{kappafig}.  Then this $\kappa_\nu$ was increased to $\kappa^1_\nu$ in the frequency range where \texttt{CO}$_2$ is absorbent, shown in red in figure \ref{kappafig}.
The corresponding temperatures and average light intensities are shown in figures \ref{tempeK2} and \ref{lightI2}.
\\\\
The main points are
\begin{itemize}
\item
For Case 1 (Visible light crossing the atmosphere downward and passing through a refractive medium for $z<Y$) Fresnel's conditions have a drastic effect on the results.
\item
For Case 2 (IR light coming from Earth and the refractive index decreases when $z>Y$) the refraction makes the medium much more absorbing. Furthermore, with a Fresnel interface, an increase in  \texttt{CO}$_2$ decreases the temperature at high altitude but increases it near the ground.
\end{itemize}

\subsection{Part III: Water and Air: Influence of the Density}
The geometry imitates an ocean of depth 1000m with an atmosphere 9000m thick. The density of water is a thousand time greater than air.  To account for this we took $\rho=10$ in water and $\rho=0.1$ in air; smaller values are problematic for the integrals.  The refractive index is as above, n=1 in water and n=0.7 in air.
Four runs were done similar to { Case 1} and { Case 2} with a Gemini $\overline{\kappa}_\nu$ and $\overline{\kappa}^1_\nu$.

Results are shown in figure \ref{tempeK3} and \ref{lightI3}. Density has a drastic effect, naturally. Notice that \texttt{CO}$_2$ increases the temperature in water and cools the atmosphere.
 
\begin{figure}
\begin{minipage} [b]{0.45\textwidth}
\begin{center}
\begin{tikzpicture}[scale=0.56]
\begin{axis}[width=13cm,height=7cm,legend style={at={(1,1)},anchor=north east}, compat=1.3,
   xmin=0,xmax=25, ymin=0, ymax=2.4,
   xlabel= {Wavelength ($\mu$m)},
  ylabel= {Absorption coefficient  $\overline\kappa$}
  ]
\addplot[thick,dashed,color=red,mark=none, mark size=1pt] table [x index=0, y index=1]{fig/kappa1.txt};
\addlegendentry{ $\overline\kappa^1$}
\addplot[thick,solid,color=black,mark=none, mark=x] table [x index=0, y index=1]{fig/kappa.txt};
\addlegendentry{ $\overline\kappa$}
\end{axis}
\end{tikzpicture}
\caption{ \footnotesize\label{kappafig} Absorption $\overline\kappa$ from the Gemini experiment, versus wavenumber ($3/\nu$). In dotted lines, the modification to construct $\overline\kappa_1$ to account for the opacity of \texttt{CO}$_2$.}
\end{center}
\end{minipage}
\hskip1.1cm
\begin{minipage} [b]{0.45\textwidth}. 
\begin{center}
\begin{tikzpicture}[scale=0.56]
\begin{axis}[legend style={at={(0.99,0.9)},anchor= east}, 
   xlabel= {Iterations}
   ylabel = {$T_-^n|_{z=500m}$ in $C^o$}
  ]
\addplot[thick,solid,color=blue,mark=*, mark size=1pt] table [x index=0, y index=1]{fig/approxFgrowth.txt};
\addlegendentry{Increasing seq.}
\addplot[thick,solid,color=red,mark=*, mark size=1pt] table [x index=0, y index=1]{fig/approxFdecrease.txt};
\addlegendentry{Decreasing seq.}
\end{axis}
\end{tikzpicture}
\caption{ \footnotesize \label{convergefig} Convergence of the temperature at altitude 300m during the iterations. In solid line when it is started with $T^0=0$, in dashed line when the initial temperature is 180°C. Notice the monotonicity of both curves.
}\end{center}
\end{minipage}
\end{figure}

\begin{figure}[htbp]
\begin{minipage} [b]{0.45\textwidth}
\begin{center}
\begin{tikzpicture}[scale=0.7]
\begin{axis}[legend style={at={(0.97,0.9)},anchor= east}, compat=1.3,
   xmin=0.01, xmax=1,
   ymax=20,
   xlabel= {Altitude 10km},
  ylabel= {Temperature $^o$C}
  ]
\addplot[thick,solid,color=blue,mark=none, mark size=1pt] table [x index=0, y index=1]{fig/withF001/temperature21.txt};
\addlegendentry{with F, $\epsilon=-0.01$, up}
\addplot[thick,dashed,color=blue,mark=none, mark size=1pt] table [x index=0, y index=1]{fig/noF001/temperature21.txt};
\addlegendentry{no F, $\epsilon=-0.01$, up }
\addplot[thick,solid,color=red,mark=none, mark size=1pt] table [x index=0, y index=1]{fig/withF03/temperature121.txt};
\addlegendentry{with F, $\epsilon=-0.3$, down}
\addplot[thick,dashed,color=red,mark=none, mark size=1pt] table [x index=0, y index=1]{fig/noF03/temperature121.txt};
\addlegendentry{no F, $\epsilon=-0.3$, down }

\end{axis}
\end{tikzpicture}
\caption{ \footnotesize \label{tempeK1}  Temperatures versus altitude  with $\overline\kappa=0.5$. In blue with $\epsilon=-0.01$ for {\bf Case 2} and in red with $\epsilon=-0.3$ for {\bf Case 1}. The solid curves are computed with Fresnel's conditions at $z=Y=0.5$ and the dashed curves are computed without them.}
\end{center}
\end{minipage}
\hskip0.5cm
\begin{minipage} [b]{0.45\textwidth}. 
\begin{center}
\begin{tikzpicture}[scale=0.7]
\begin{axis}[legend style={at={(-0.1,1.05)},anchor= west}, compat=1.3,
  ymax=25,
   ylabel= {light-intensity},
  xlabel= {wave length $\mu$m}
  ]
\addplot[thick,dashed,color=blue,mark=none, mark size=1pt] table [x index=0, y index=1]{fig/noF001/imean21Z.txt};
\addlegendentry{$10^5J_0(Z), \epsilon=-0.01$ no F}
\addplot[thick,solid,color=blue,mark=none, mark size=1pt] table [x index=0, y index=1]{fig/withF001/imean21Z.txt};
\addlegendentry{$10^7  J_0(Z), \epsilon=-0.01$ with F}
\addplot[thick,dashed,color=red,mark=none, mark size=1pt] table [x index=0, y index=1]{fig/noF03/imean1210.txt};
\addlegendentry{$10^5J_0(0), \epsilon=-0.3$ no F}
\addplot[thick,solid,color=red,mark=none, mark size=1pt] table [x index=0, y index=1]{fig/withF03/imean1210.txt};
\addlegendentry{$10^7J_0(0), \epsilon=-0.3$ with F}
\addplot[thick,solid,color=brown,mark=none, mark size=1pt] table [x index=0, y index=2]{fig/withF03/imean1210.txt};
\addlegendentry{$10^7K_0(0), \epsilon=-0.3$ with F}

\addplot[thick,dashed,color=brown,mark=none, mark size=1pt] table [x index=0, y index=2]{fig/noF03/imean1210.txt};
\addlegendentry{$10^7K_0(0), \epsilon=-0.3$ no F}
\end{axis}
\end{tikzpicture}
\caption{ \footnotesize \label{lightI1} Total light intensity $J_0$  and polarization $K_0$ versus wave length at $z=0$ or  $z=Z$ for the computations of figure \ref{tempeK1}.}
\end{center}
\end{minipage}
\end{figure}


\begin{figure}[htbp]
\begin{minipage} [b]{0.45\textwidth}
\begin{center}
\begin{tikzpicture}[scale=0.7]
\begin{axis}[legend style={at={(0.97,0.9)},anchor= east}, compat=1.3,
   xmin=0.01, xmax=1,
   ymax=0,
   xlabel= {Altitude 10km},
  ylabel= {Temperature $^o$C}
  ]
\addplot[thick,solid,color=blue,mark=none, mark size=1pt] table [x index=0, y index=1]{fig/CO2/temperature1.txt};
\addlegendentry{Gemini}
\addplot[thick,dashed,color=blue,mark=none, mark size=1pt] table [x index=0, y index=1]{fig/CO2/temperature11.txt};
\addlegendentry{Gemini with CO2}
\addplot[thick,solid,color=red,mark=none, mark size=1pt] table [x index=0, y index=1]{fig/CO3/temperature1.txt};
\addlegendentry{Gemini no F}
\addplot[thick,dashed,color=red,mark=none, mark size=1pt] table [x index=0, y index=1]{fig/CO3/temperature11.txt};
\addlegendentry{Gemini with CO2 no F}
\end{axis}
\end{tikzpicture}
\caption{ \footnotesize \label{tempeK2} {\bf Case 2}. Effect of CO2 on the temperature in the presence of a Fresnel interface with $\epsilon=-0.3$.}
\end{center}
\end{minipage}
\hskip0.5cm
\begin{minipage} [b]{0.45\textwidth}. 
\begin{center}
\begin{tikzpicture}[scale=0.7]
\begin{axis}[legend style={at={(-0.2,0.8)},anchor= west}, compat=1.3,
 ymax=80,
   ylabel= {light-intensity},
  xlabel= {wave length $\mu$m}
  ]
\addplot[thick,solid,color=blue,mark=none, mark size=1pt] table [x index=0, y index=1]{fig/CO2/imean1Z.txt};
\addlegendentry{$10^5J_0(Z)$, Gemini}
\addplot[thick,solid,color=brown,mark=none, mark size=1pt] table [x index=0, y index=2]{fig/CO2/imean1Z.txt};
\addlegendentry{$10^7K_0(Z)$, Gemini}
\addplot[thick,dashed,color=blue,mark=none, mark size=1pt] table [x index=0, y index=1]{fig/CO2/imean11Z.txt};
\addlegendentry{$10^5J_0(Z)$, Gemini+CO2}
\addplot[thick,dashed,color=brown,mark=none, mark size=1pt] table [x index=0, y index=2]{fig/CO2/imean11Z.txt};
\addlegendentry{$10^7K_0(Z)$, Gemini+CO2}
\end{axis}
\end{tikzpicture}
\caption{ \footnotesize \label{lightI2} {\bf Case 2}. Total light intensity $J_0$  and polarized $K_0$ versus wave length at altitude Z=10km. }
\end{center}
\end{minipage}
\end{figure}
\begin{figure}[htbp]
\begin{minipage} [b]{0.45\textwidth}
\begin{center}
\begin{tikzpicture}[scale=0.7]
\begin{axis}[legend style={at={(0.97,0.9)},anchor= east}, compat=1.3,
   xmin=0.01, xmax=0.5,
   ymax= 50,ymin=-50,
   xlabel= {Altitude 10km},
  ylabel= {Temperature $^o$C}
  ]
\addplot[thick,solid,color=blue,mark=none, mark size=1pt] table [x index=0, y index=1]{fig/CO4/temperature1.txt};
\addlegendentry{Gemini Up}
\addplot[thick,dashed,color=blue,mark=none, mark size=1pt] table [x index=0, y index=1]{fig/CO4/temperature11.txt};
\addlegendentry{Gemini Up with CO2}
\addplot[thick,solid,color=red,mark=none, mark size=1pt] table [x index=0, y index=1]{fig/CO5/temperature101.txt};
\addlegendentry{Gemini Down}
\addplot[thick,dashed,color=red,mark=none, mark size=1pt] table [x index=0, y index=1]{fig/CO5/temperature111.txt};
\addlegendentry{Gemini Down with CO2}
\end{axis}
\end{tikzpicture}
\caption{ \footnotesize \label{tempeK3} {\bf Case 3} with infrared coming from $z=0$(red) and (blue) with Sun rays coming from $Z$. Dotted curves display the effect of CO2 on the temperature in the presence of a Fresnel interface with $\epsilon=-0.3$.(Results are not shown for $z>0.5$) because the temperature hardly changes.}
\end{center}
\end{minipage}
\hskip0.5cm
\begin{minipage} [b]{0.45\textwidth}. 
\begin{center}
\begin{tikzpicture}[scale=0.7]
\begin{axis}[legend style={at={(0.6,1.5)},anchor= north}, compat=1.3,
  xmax=30,
   ylabel= {light-intensity},
  xlabel= {wave length $\mu$m}
  ]
\addplot[thick,solid,color=blue,mark=none, mark size=1pt] table [x index=0, y index=1]{fig/CO4/imean1Z.txt};
\addlegendentry{$10^5J_0(Z)$, Gemini}
\addplot[thick,solid,color=brown,mark=none, mark size=1pt] table [x index=0, y index=2]{fig/CO4/imean1Z.txt};
\addlegendentry{$10^7K_0(Z)$, Gemini}
\addplot[thick,dashed,color=blue,mark=none, mark size=1pt] table [x index=0, y index=1]{fig/CO4/imean11Z.txt};
\addlegendentry{$10^5J_0(Z)$, Gemini+CO2}
\addplot[thick,dashed,color=brown,mark=none, mark size=1pt] table [x index=0, y index=2]{fig/CO4/imean11Z.txt};
\addlegendentry{$10^7K_0(Z)$, Gemini+CO2}
\end{axis}
\end{tikzpicture}
\caption{ \footnotesize \label{lightI3} {\bf Case 3}. Total light intensity $J_0$  and polarized $K_0$ versus wave length at altitude Z. }
\end{center}
\end{minipage}
\end{figure}

\subsection{Computation of $I$ and $Q$ versus $z$ and $\mu$}

For each $\nu$ at which  $I$ and $Q$ are desired we use \eqref{lqn} with \eqref{SS}. The computation is fast but some of the integrals are singular, so it needs to be implemented with care. The absorption is constant$\overline\kappa_\nu=0.5$, the density is $1-0.75 z$. Uniform scattering is applied with $a_s=0.7$. The results are displayed at $\nu=0.1435$.

 {\bf Case 2} is computed, first
with $n(z)=1-0.3\One_{z>Y}$ but no Fresnel conditions  added.  Results are in figures \ref{Izmu2}, \ref{Kzmu2}. The jumps are mostly due to the multiplication by $n^2(z)$ in $I=n^2(z)\tilde I$; $\tilde I$ has a much smaller jump at $z=Y$.

Then, the same computation is done with Fresnel Conditions added. Results are in figures \ref{Izmu3}, \ref{Kzmu3}. Notice that the jump is bigger.

%

\begin{figure}[htbp]
\begin{minipage} [b]{0.47\textwidth}
\centering
\begin{tikzpicture}[scale=0.8]
\begin{axis}[ legend style={at={(1,1)},anchor=north east}, compat=1.3,xlabel= {$z$},ylabel= {$\mu$}]
 \addplot3[surf,fill opacity=0.75] table [ ] {fig/nonewithN/gnuplotI.txt};
\addlegendentry{ $10^5\cdot I(z,\mu)$}
\end{axis}
\end{tikzpicture}
\caption{Case 2, $\overline\kappa_\nu=0.5$: Light intensity when $n(z)=1-0.3\One_{z>Y}$.  No Fresnel condition added.}
\label{Izmu2} 
\end{minipage}
\hskip 0.25cm
\begin{minipage} [b]{0.47\textwidth}
\centering
\begin{tikzpicture}[scale=0.8]
\begin{axis}[legend style={at={(1,1)},anchor=north east}, compat=1.3,xlabel= {$z$},ylabel= {$\mu$}]
 \addplot3[surf,fill opacity=0.5] table [ ] {fig/nonewithN/gnuplotK.txt};
\addlegendentry{  $10^5\cdot Q(z,\mu)$}
\end{axis}
\end{tikzpicture}
\caption{Case 2, $\overline\kappa_\nu=0.5$: Polarization when $n(z)=1-0.3\One_{z>Y}$. No Fresnel condition added.}
\label{Kzmu2}
\end{minipage}%
\end{figure}

\begin{figure}[htbp]
\begin{minipage} [b]{0.47\textwidth}
\centering
\begin{tikzpicture}[scale=0.8]
\begin{axis}[ legend style={at={(1,1)},anchor=north east}, compat=1.3,xlabel= {$z$},ylabel= {$\mu$}]
 \addplot3[surf,fill opacity=0.75] table [ ] {fig/nonewithF/gnuplotI.txt};
\addlegendentry{ $10^5\cdot I(z,\mu)$}
\end{axis}
\end{tikzpicture}
\caption{Case 2: Light intensity when $n(z)=1-0.3\One_{z>Y}$, uniform scattering with $a_s=0.7$.  With Fresnel conditions at $z=Y$.}
\label{Izmu3} 
\end{minipage}
\hskip 0.25cm
\begin{minipage} [b]{0.47\textwidth}
\centering
\begin{tikzpicture}[scale=0.8]
\begin{axis}[legend style={at={(1,1)},anchor=north east}, compat=1.3,xlabel= {$z$},ylabel= {$\mu$}]
 \addplot3[surf,fill opacity=0.5] table [ ] {fig/nonewithF/gnuplotK.txt};
\addlegendentry{  $10^5\cdot Q(z,\mu)$}
\end{axis}
\end{tikzpicture}
\caption{Polarization when $n(z)=1-0.3\One_{z>Y}$, uniform scattering with $a_s=0.7$.  With Fresnel conditions at $z=Y$.}
\label{Kzmu3}
\end{minipage}%
\end{figure}
%

\section{Precision}
The iterations converge rather fast and the solution can be bounded from above and below by the decreasing and increasing sequences.  The Newton iterations to compute the temperature from the knowledge of $I$ can also be driven to machine precision with a small number of iterations because $T\mapsto B_\nu(T)$ is strictly increasing.
The bottleneck is the precision to compute integrals such as
\[\ds
J_k(z,z') = \int_{[0,1]\cap M}\mu^{k-1}\frac{\exp\left( -\kappa_\nu\int_z^{z'}\kappa(y)\left(1-(1-\mu^2)\frac{n^2(z)}{n^2(y)}\right)_+^{-\frac12}dy\right)}{\left(1-(1-\mu^2)\frac{n^2(z)}{n^2(z')}\right)^{\frac12}}\d\mu
\]
with $M=\{\mu~:~1-(1-\mu^2)\frac{n^2(z)}{n^2(z')}>0\}$.  There is a singularity at $\mu=\sqrt{1-\frac{n^2(z')}{n^2(z)}}$ if non-negative.

We use a quadrature formula at $\mu^j=(j\delta\mu)^2$ if $\mu^j<\mu^*$ and $\mu^j=j\delta\mu$ if $\mu^j>\mu^*$.

When $n$ is constant, $J_k(0,z')$ is the exponential integral $E_k(\kappa_\nu\int_0^{z'}\kappa(y)d y)$ for which there is a very precise approximated formula when $\kappa_\nu\kappa(y)$ is not large \cite{ABR}.
With $\kappa_\nu=0.5$ and $\kappa(z)=1-z/2$, figure \ref{convergeExp} displays the precision obtained on $E_1,E_3,E_5$ when $\delta\mu=0.02,0.01,0.005$ and $\mu^*=0.1$. Already with $\mu=0.01$ the relative precision is less than $1\%$.The integral in the exponential is computed with a fixed increment $\delta z=1/60$.

When $n$ is not constant we can only observe the convergence towards the value obtained with a very small $\delta\mu$ and $\delta z$, as shown in figure \ref{convergeExp2}. The convergence is not monotone in $\delta\mu$, so it is hard to say if the parameters are small enough.
\begin{figure}[htbp]
\begin{minipage} [b]{0.45\textwidth}. 
\begin{center}
\begin{tikzpicture}[scale=0.7]
\begin{axis}[legend style={at={(0.99,0.9)},anchor= east}, 
   ymin=0., ymax=4,
   xlabel= {Altitude},
   ylabel = {Relative pointwise error in \%.}
  ]
\addplot[thick,solid,color=blue,mark=*, mark size=1pt] table [x index=0, y index=1]{fig/test.txt};
\addlegendentry{$100|E_1(z)-E_{1h}(z)|/E_1(z)$.}
\addplot[thick,solid,color=red,mark=*, mark size=1pt] table [x index=0, y index=2]{fig/test.txt};
\addlegendentry{$100|E_3(z)-E_{3h}(z)|/E_3(z)$.}
\addplot[thick,solid,color=green,mark=*, mark size=1pt] table [x index=0, y index=3]{fig/test.txt};
\addlegendentry{$100|E_5(z)-E_{5h}(z)|/E_5(z)$.}
\end{axis}
\end{tikzpicture}
\caption{ \footnotesize \label{convergeExp} Convergence of the approximate exponential integrals $E_{1h},E_{3h},E_{5h}$ to $E_{1},E_{3},E_{5}$ versus $\d\mu=0.02,0.01,0.05$. The pointwise relative errors are plotted versus $z$. At $\d\mu=0.01$ the 3 relative errors are below 1\%. $\d\mu=0.005$ does not improve the precision.
}\end{center}
\end{minipage}
\hskip1.5cm
\begin{minipage} [b]{0.45\textwidth}. 
\begin{center}
\begin{tikzpicture}[scale=0.7]
\begin{axis}[legend style={at={(0.99,0.9)},anchor= east}, 
   ymin=0., ymax=0.003,
   xlabel= {Altitude},
   ylabel = {Absolute pointwise error in \%.}
  ]
%
\addplot[thick,solid,color=red,mark=*, mark size=1pt] table [x index=0, y index=2]{fig/ntest.txt};
\addlegendentry{$|E_{1\frac1{100}}(z)-E_{1\frac1{800}}(z)|$}
\addplot[thick,solid,color=green,mark=*, mark size=1pt] table [x index=0, y index=3]{fig/ntest.txt};
\addlegendentry{$|E_{1\frac1{200}}(z)-E_{1\frac1{800}}(z)|$}
\addplot[thick,solid,color=blue,mark=*, mark size=1pt] table [x index=0, y index=4]{fig/ntest.txt};
\addlegendentry{$|E_{1\frac1{400}}(z)-E_{1\frac1{800}}(z)|$}
%
%
\addplot[thick,solid,color=pink,mark=*, mark size=1pt] table [x index=0, y index=2]{fig/ntest5.txt};
\addlegendentry{$|E_{5\frac1{100}}(z)-E_{5\frac1{800}}(z)|$}
\addplot[thick,solid,color=magenta,mark=*, mark size=1pt] table [x index=0, y index=3]{fig/ntest5.txt};
\addlegendentry{$|E_{5\frac1{200}}(z)-E_{5\frac1{800}}(z)|$}
\addplot[thick,solid,color=black,mark=*, mark size=1pt] table [x index=0, y index=4]{fig/ntest5.txt};
\addlegendentry{$|E_{5\frac1{400}}(z)-E_{5\frac1{800}}(z)|$}

\end{axis}
\end{tikzpicture}
\caption{ \footnotesize \label{convergeExp2} Convergence of the approximate exponential integrals $E_{1h}$,$E_{5h}$ versus $\d\mu=0.01,0.005,0.0025$ when $n(z)=1\pm 0.3\One_{z>0.5}$. The pointwise absolute errors are plotted versus $z$. In the case of $E_{1h}$ both curves corresponding to the 2 signs in $n$ are plotted. }
\end{center}
\end{minipage}
\end{figure}

\section{Conclusion}
In this article the methodology developed in \cite{FGOP3} for the numerical solution of the VRTE has been extended to include the Fresnel Conditions at an interface of discontinuity of the refractive index. While the solution of the equation \eqref{fundamental} given by Pomraning and Chandrasekhar in \cite{CHA} \cite{POM} do give a jump of the Stokes vector and the temperature at the discontinuity, the amplitude of the jump is not the same as the one given by the Fresnel Conditions.

In principle the method is not hard to program (500 lines of \texttt{C++}) and the execution time is a few seconds; however, the formulas are complex and the probability of having bugs cannot be ruled out. Yet this is a very fast method to solve the VRRTE in all generalities for the coefficients.

As before we have measured numerically the effect of a change on the  absorption due to \texttt{CO}$_2$. Although  preliminary, the  conclusion reached in our earlier studies are also valid here when there is a change of refraction index as in water and air: the effect of \texttt{CO}$_2$ on the infrared radiation from Earth heats up the region near the ground and cools it in high altitude. 

Generalization to 3D as in \cite{JCP}) and \cite{JCP2} for a non-stratified atmosphere is possible but the complexity of handling curved surfaces of refractions is high, as for geometrical optics.

\bibliographystyle{plain}

\bibliography{references}

\begin{thebibliography}{10}

\bibitem{ABR}
M.~Abramowitz and I.~Stegun.
\newblock {\em Handbook of Mathematical Functions}.
\newblock Dover Publications, Washington D.C., 1972.

\bibitem{graphics}
M.~Ament, C.~Bergmann, and D.~Weiskopf.
\newblock Refractive radiative transfer equation.
\newblock {\em ACM Transactions on Graphics}, 33(2):2, March 2014.

\bibitem{zhang}
Yong~Zhang andf Hong-Liang~Yi and He-Ping Tan.
\newblock Lattice boltzmann method for short-pulsed laser transport in a
  multi-layered medium.
\newblock {\em Journal of Quantitative Spectroscopy and Radiative Transfer},
  155(April):75--89, 2015.

\bibitem{BEN}
Ph. Ben-Abdallah.
\newblock When the space curvature dopes the radiant intensity.
\newblock {\em J. Opt. Soc. Am. B}, 19(8), 2002.

\bibitem{measureN}
H.~Bussey and G.~Birnbaum.
\newblock Measurement of variations in atmospheric refractive index xith an
  airborne microwave refractometer.
\newblock {\em Journal of National Research of the National Bureau of
  Standards}, 51(4):171--178, 1953.

\bibitem{wang}
Y.H. Yang Y. Zhang K. Yue X.X.~Zhang C.H.~Wang, Y.Y.~Feng.
\newblock Chebyshev collocation spectral method for vector radiative transfer
  equation and its applications in two-layered media.
\newblock {\em J. Quant. Spectrosc. Radiat. Transf..}, 243:106822, 2020.

\bibitem{CHA}
S.~Chandrasekhar.
\newblock {\em {Radiative Transfer}}.
\newblock Clarendon Press, Oxford, 1950.

\bibitem{DUF}
J.~Dufresne, V.~Eymet, C.~Crevoisier, and J.~Grandpeix.
\newblock Greenhouse effect: The relative contributions of emission height and
  total absorption.
\newblock {\em Journal of Climate, American Meteorological Society},
  33(9):3827--3844, 2020.

\bibitem{garcia}
R.~D.M. Garcia.
\newblock Radiative transfer with polarization in a multi-layer medium subject
  to fresnel boundary and interface conditions.
\newblock {\em Journal of Quantitative Spectroscopy and Radiative Transfer},
  (115):28--45, 2013.

\bibitem{JCP}
F.~Golse, F.~Hecht, O.~Pironneau, D.~Smetz, and P.-H. Tournier.
\newblock Radiative transfer for variable 3d atmospheres.
\newblock {\em J. Comp. Physics}, 475(111864):1--19, 2023.

\bibitem{FGOP3}
F.~Golse and O.~Pironneau.
\newblock Stratified radiative transfer in a fluid and numerical applications
  to earth science.
\newblock {\em SIAM Journal on Numerical Analysis}, 60(5):2963--3000, 2022.

\bibitem{dehann}
J.~De Haan, P.~Bosma, and J.~Hovenier.
\newblock The adding method for multiple scattering computations of polarized
  light.
\newblock {\em Astron Astrophys}, 183:371--391, 1987.

\bibitem{Liu}
L.-H. Liu.
\newblock Finite volume method for radiation heat transfer in graded index
  medium.
\newblock {\em Journal Of Thermophysics And Heat Transfer}, 20(1):59--66, Jan
  2006.

\bibitem{OP2024}
O.~Pironneau.
\newblock Numerical simulation of polarized light and temperature in a
  stratified atmosphere with a slowly varying refractive index.
\newblock {\em Pure and Applied Functional Analysis}, Special issue dedicated
  to Roger Temam, 2024.

\bibitem{OP2023}
O.~Pironneau.
\newblock Numerical simulation of polarized light with \text{R}ayleigh
  scattering in a stratified atmosphere.
\newblock {\em Pure and Applied Functional Analysis}, Special issue dedicated
  to Luc Tartar, 2024.

\bibitem{JCP2}
O.~Pironneau and P.-H. Tournier.
\newblock Reflective conditions for radiative transfer in integral form with
  h-matrices.
\newblock {\em Journal of Computational Physics}, 495(112531):1--14, 2023.

\bibitem{POM}
G.~Pomraning.
\newblock {\em The equations of Radiation Hydrodynamics}.
\newblock Pergamon Press, NY, 1973.

\bibitem{POM2}
G.~Pomraning and B.~Ganapol.
\newblock Simplified radiative transfer for combined rayleigh and isotropic
  scattering.
\newblock {\em The Astrophysical Journal}, 498:671--688, 1998.

\bibitem{tan}
C.H. Wang and H.P. Tan.
\newblock Discontinuous finite element method for vector radiative transfer.
\newblock {\em J Quant Spectrosc Radiat Transf.}, 189:383--397, 2017.

\bibitem{weng}
F.~Weng.
\newblock A multi-layer discrete-ordinate method for vector radiative transfer
  in a vertically-inhomogeneous, emitting and scattering atmosphere--i. theory.
\newblock {\em J Quant Spectrosc Radiat Transfer}, 47:19--33, 1992.

\end{thebibliography}

\end{document}